\documentclass[a4paper,12pt,fleqn]{article}
\usepackage[latin1]{inputenc}  \usepackage[T1]{fontenc}
\usepackage{lscape}
\usepackage{tabularx}     
\usepackage{pstricks}    
\usepackage{rotating}
\usepackage{amsfonts,amsmath,amsthm,amssymb,dsfont}
\usepackage{marginnote}
\usepackage{rotate}
\usepackage[arrow,matrix,curve]{xy} 
\usepackage{tikz}
\usepackage{tikz-cd}	
\title{Discrepancy Properties and Conjugacy Classes of Interval Exchange Transformations}
\author{Christian Wei\ss}
\date{\today}
\makeatletter
\newcommand{\subjclass}[2][2020]{%
	\let\@oldtitle\@title%
	\gdef\@title{\@oldtitle\footnotetext{#1 \emph{Mathematics Subject Classification.} #2}}%
}
\newcommand{\keywords}[1]{%
	\let\@@oldtitle\@title%
	\gdef\@title{\@@oldtitle\footnotetext{\emph{Keywords.} #1.}}%
}
\makeatother
\subjclass{37E05, 11J71, 11K38, 11B50}	
\keywords{Interval exchange transfomrations; Conjugacy classes; Low-discpreancy sequences; Fundamental discontinuities}

\newtheorem{thm}{Theorem}[section]
\newtheorem{defi}[thm]{Definition}

\newtheorem{lem}[thm]{Lemma}
\newtheorem{prop}[thm]{Proposition}
\newtheorem{cor}[thm]{Corollary}

\newtheorem{exa}[thm]{Example}

\newtheorem{innercustomgeneric}{\customgenericname}
\providecommand{\customgenericname}{}
\newcommand{\newcustomtheorem}[2]{%
	\newenvironment{#1}[1]
	{%
		\renewcommand\customgenericname{#2}%
		\renewcommand\theinnercustomgeneric{##1}%
		\innercustomgeneric
	}
	{\endinnercustomgeneric}
}

\newcustomtheorem{customthm}{Theorem}
\newcustomtheorem{customcor}{Corollary}
\newcustomtheorem{customprop}{Proposition}

\newcommand{\RR}{{\mathbb{R}}}

\newcommand{\NN}{{\mathbb{N}}}
\newcommand{\QQ}{{\mathbb{Q}}}

\makeindex

\begin{document} 

\maketitle

\begin{abstract} Interval exchange transformations are typically uniquely ergodic maps and therefore have uniformly distributed orbits. Their degree of uniformity can be measured in terms of the star-discrepancy. Few examples of interval exchange transformations with low-discrepancy orbits are known so far and only for $n=2,3$ intervals, there are criteria to completely characterize those interval exchange transformations. In this paper, it is shown that having low-discrepancy orbits is a conjugacy class invariant under composition of maps. To a certain extent, this approach allows us to distinguish interval exchange transformations with low-discrepancy orbits from those without. For $n=4$ intervals, the classification is almost complete with the only exceptional case having monodromy invariant $\rho = (4,3,2,1)$. This particular monodromy invariant is discussed in detail.
\end{abstract}


\section{Introduction} Interval exchange transformations (IETs) are an important kind of generalization of circle rotations and are therefore a widely considered class of discrete dynamical systems, see e.g. \cite{HK02}. They act on the unit interval by cutting it into $n \in \mathbb{N}$ subintervals and permuting these subintervals. Hence, they are piecewise linear functions. Besides their rich dynamics, their special importance lies amongst others in their connection to translation flows, see e.g. \cite{Via06}, and moduli spaces, see e.g. \cite{Yoc06}. In the simplest case of $n=2$ intervals, interval exchange transformations are circle rotations and their orbits are the famous (shifted) Kronecker sequences.\\[12pt]
A seminal result, which was independently proved in \cite{Mas82} and \cite{Vee78}, states that almost every interval exchange transformation is uniquely ergodic. Thus, they typically have uniformly distributed orbits. However, this does not mean that the orbits are necessarily low-discrepancy sequences, i.e. roughly speaking as uniformly distributed as possible, which only occurs on rare occasions. The present article is dedicated to the purpose to better understand when interval exchange transformations possess low-discrepancy orbits.\\[12pt] 
Recall that for a sequence $P=(x_i)_{i=1}^\infty$ in $[0,1)$, the \textbf{star-discrepancy} of its first $N$ points is defined by
$$D^*_N(P) := \sup_{B \subset [0,1)} \left| \frac{|\left\{ x_i | 1 \leq i \leq N \right\} \cap B|}{N} - \lambda(B) \right|,$$
where the supremum is taken over all intervals $B = [0,a) \subset [0,1)$ anchored at zero and $\lambda(\cdot)$ denotes the Lebesgue measure. If $D_N^*(P)$ satisfies
$$D_N^*(P) = O(N^{-1}(\log(N)))$$
then $P$ is called a \textbf{low-discrepancy sequence}. Indeed, this is the best possible order of convergence to zero by the work of Schmidt, see \cite{Sch72}.  The precise best possible value of the constant $c$ with $D_N^*(P) \leq c N^{-1} \log(N)$ for all $N \in \NN$ and a sequence $P$ is still unknown (see e.g. \cite{Lar14}). The concept of star-discrepancy is closely related to the construction of optimal integration rules via the Koksma-Hlawka inequality. For more details, in particular on higher dimensions, we refer the reader to \cite{DP10}.\\[12pt]
A complete classification of low-discrepancy orbits for IETs with $n=2$ intervals goes back to Behnke: the low-discrepancy property relies on the continued fraction expansion of the angle of rotation. These orbits are some of the classical examples of low-discrepancy sequences, i.e. (shifted) Kronecker sequences, see e.g. \cite{Nie92}. 
\begin{thm} \label{thm:kronecker} (\cite{DT97}, Corollary 1.65) \label{thm:n=2} Let $z \in \RR \setminus \QQ$ with continued fraction expansion $z =[a_0,a_1,a_2,\ldots]$. Then the sequence $(\left\{ n z \right\})_{n \geq 0}$, where $\left\{ \cdot \right\}$ denotes the fractional part, has low-discrepancy if and only if the Ces\`aro mean
	$$a_m^{(1)}(z) = \frac{1}{m} \sum_{j=1}^m a_j$$
is a bounded sequence.  	
\end{thm}  
While a theorem of Khintchine implies that the set $$\left\{ z \in [0,1) \  | \ z \notin \QQ, a_m^{(1)}(z) \ \textrm{is a bounded squence} \right\}$$ has Lebesgue measure zero, see \cite{Khi63}, it has Hausdorff dimension one due to the Jarnik-Besicovitch Theorem, see \cite{Jar31}, \cite{Bes34}. The same observation essentially holds true for $n=3$ intervals. In this case, an (admissible) interval exchange transformation is not a rotation if and only if the ordering of the intervals is inverted by the map. In other words, the so-called monodromy invariant, which describes the ordering of the subintervals after applying the interval exchange transformation, is $\rho = (3,2,1)$ then, compare with \cite{Via06}.\footnote{The general definition of the monodromy invariant is given in Section~\ref{sec:inequ}. Here we assume that the subintervals before rotation are labeled in increasing ordering.} 
\begin{thm} (\cite{Wei19}, Theorem 3.1) \label{thm:n=3} Let $\lambda_1, \lambda_2, \lambda_3$ be the lengths of three intervals with $\lambda_1 + \lambda_2 + \lambda_3 = 1$. Let $f: [0,1) \to [0,1)$ denote the interval exchange transformation, which inverts the ordering of the three intervals. Then $f$ yields a low-discrepancy sequence $(f^n(y))_{n \geq 0}$ for all $y \in [0, 1)$ if and only if $\frac{\lambda_2+\lambda_3}{1 + \lambda_2}$ is irrational and its continued fraction expansion has bounded Ces\`aro mean.  
\end{thm}
Recall that for a map $f: [0,1) \to [0,1)$, its iteration is inductively defined by $f^1:=f$ and $f^{i+1} := f(f^i)$. Together, Theorem~\ref{thm:n=2} and Theorem~\ref{thm:n=3} completely characterize all interval exchange transformation with $n \in \left\{ 2,3 \right\}$ intervals that posses low-discrepancy orbits. In the case of $n \geq 4$ intervals, there is much less known. It is a consequence of \cite{Zor97} that given a monodromy invariant Lebesgue-almost all interval exchange do not have low-discrepancy orbits for $n \geq 4$ as well. In contrast, an abstract criterion for identifying interval exchange transformations with low-discrepancy orbits was derived in \cite{Dou99}. It involves the constructive geometric definition of systems of rank one, see also \cite{Fer97}, and can therefore hardly be applied in practice to check if a given IET has a low-discrepancy orbit. To the best of our knowledge, the only explicit examples of interval exchange transformations with $n \geq 4$ intervals having low-discrepancy orbits have been found in \cite{Wei19}.\\[12pt] 
In this paper, we aim to contribute to a better understanding of the situation. The idea is to consider conjugacy classes of the group of all interval exchange transformations, where the group action is given by functional composition. The following theorem is a key observation for this novel approach, which allows applications on the distributional properties of IET orbits. Its proof will be given in Section~\ref{sec:inequ}.
\begin{thm} \label{thm:main_thm} Let $P=(x_i)_{i=1}^\infty$ be a sequence of points in $[0,1)$. Furthermore, let $f$ be an (arbitrary) interval exchange transformation with $n$ intervals and denote the sequence $(f(x_i))_{i=1}^\infty$ by $P^*$. For every $N \in \mathbb{N}$ then
$$\frac{1}{n} D_N^*(P)  \leq D_N^*(P^*) \leq n D_N^*(P)$$
holds.
\end{thm}
Note that the second inequality follows from the first by considering $f^{-1}$ instead of $f$. By applying the well-known formula (see e.g. \cite{Nie92}, Theorem 2.6)
$$D_N^*(P) = \frac{1}{2N} + \max_{i=1,2,\ldots,N} \left| x_i - \frac{2i-1}{N} \right|$$
for the star-discrepancy of an ordered set $P=\left\{x_1,x_2,\ldots,x_N\right\} \subset [0,1)$ with $x_1 \leq x_2 \leq \ldots \leq x_N$, the boundary in Theorem~\ref{thm:main_thm} can be seen to be sharp in the following cases.  
\begin{itemize}
		\item For $n=1$, the interval exchange transformation is the identity and the claim follows trivially.
		\item For $n=2$ and $N$ arbitrary, choose 
		$$x_i = \frac{1}{N} + (i-1) \frac{N-2}{N(N-1)}$$ 
		for $i=1,\ldots,N$. Then $D^*_N(x_1,\ldots,x_N) = 1/N$. Now let $f$ exchange the two intervals $[0,1/N)$ and $[1/N,1)$. Then $D^*_N(f(x_1),\ldots,f(x_N)) = 2/N$.
		\item For $n=2N$, we choose the starting set $P$ as $x_i = (2i-1)/2N$ for $i=1,\ldots N$ . This set has star-discrepancy $1/(2N)$ which is the lowest possible value that can be achieved. Now let $1/N > \varepsilon > 0$ be arbitrary. For $i = 1, \ldots, 2N$ we define the following intervals 
		$$I_i = \begin{cases} [0, x_1 - \varepsilon/2) & i = 1\\ 
		[x_n - \varepsilon/2, x_n + \varepsilon/2) & i = 2n, n=1, \ldots, N\\
		[x_{n} + \varepsilon/2, x_{n+1} - \varepsilon/2) & i = 2n+1, n=1, \ldots, N-1\\
		[x_{N} - \varepsilon/2, 1) & i = 2N
		\end{cases}$$ 
		and choose the monodromy invariant $\rho$ as
		$$\rho = \begin{cases} 2i & i \leq N \\ i - N + (i-N-1)\cdot 2& N < i \leq 2N.
		\end{cases}$$
		Then the star-discrepancy converges $D_N^*(f(x_1),\ldots,f(x_N)) \to 1$ as $\epsilon \to 0$. 
\end{itemize}
On the other hand, it is trivial that the given boundary cannot be sharp for interval exchange transformations with $n>2N$ intervals because $\frac{1}{2N} \leq D_N^*(x_1,\ldots,x_N) \leq 1$. By applying Theorem~\ref{thm:main_thm} to an interval exchange transformation $g$ and realizing that $(gfg^{-1})^i(x) = gf^i(g^{-1}(x))$, the following corollary is an immediate consequence of Theorem~\ref{thm:main_thm}.
\begin{cor} \label{cor:ld_conjugacy} Let $f,g$ be two interval exchange transformations and let $x \in [0,1)$ be arbitrary. Then $(f^i(g^{-1}(x)))_{i=1}^\infty$ is a low-discrepancy sequence if and only if $(\left(gfg^{-1}\right)^i(x))_{i=1}^\infty$ is a low-discrepancy sequence. 
\end{cor}
This result is particularly helpful in the case of $n=4$ intervals, where we can show that an (admissible) interval exchange transformation with $n=4$ intervals and monodromy invariant $\rho \neq (4,3,2,1)$ is always conjugate to an interval exchange transformation with at most $3$ intervals (Theorem~\ref{thm:old_forms_n4}). We call these interval exchange transformations \textbf{old transformations}. Theorems~\ref{thm:kronecker} and Theorem~\ref{thm:n=3} can then be used to determine if an old transformation has a low-discrepancy orbit or not. This constitutes a significant step towards a complete and practically applicable (in contrast to the systems of rank one viewpoint) classification of IETs that have low-discrepancy orbits. Moreover, it limits the number of monodromy invariants to consider when searching for potential new low-discrepancy sequences generated by IETs.\\[12pt] 
To complete the picture for $n=4$ intervals, Proposition~\ref{prop:newtrans} gives a sufficient condition under which an interval exchange transformation with $\rho = (4,3,2,1)$ cannot be achieved by conjugation of an interval exchange transformation with a lower number of intervals. Such an interval exchange transformation is called a \textbf{new transformation}. An essential tool in this context is the work of Bernazzani in \cite{Ber18}.\\[12pt]
The conjugation method proposed in the present paper can also be applied in the case of $n > 4$ intervals. However, it is farer from yielding a complete classification then. For instance $50$ of $71$ (admissible) monodromy invariants in the case of $n = 5$ intervals can this way be excluded to give new transformations. Furthermore, it is shown in Example~\ref{exa:strongly_separating} that there are old transformations with $n=5$ intervals having one of the $21$ monodromy invariants we could not directly exclude.  

\section{Discrepancy Properties of Conjugacy Classes} \label{sec:inequ}


\paragraph{Interval Exchange Transformations.} 
Let  $\left\{ I_\alpha | \alpha \in \mathcal{A} \right\}$ be a finite partition of the unit interval $[0,1)$ into sub-intervals indexed by the finite alphabet $\mathcal{A}=\left\{1,\ldots,n\right\}$. An \textbf{interval exchange transformation} is a map $f: [0,1) \to [0,1)$ which is a translation on each subinterval $I_\alpha$.  It is determined by its combinatorial data and its length data. The \textbf{combinatorial data} consists of two bijections $\pi_0, \pi_1: \mathcal{A} \to \mathcal{A}$, and the \textbf{length data} are numbers $(\lambda_\alpha)_{\alpha \in \mathcal{A}}$ with $\lambda_\alpha > 0$ and $1 = \sum_{\alpha \in \mathcal{A}} \lambda_\alpha$. The number $\lambda_\alpha$ is the length of the subinterval $I_\alpha$ and the pair $\pi = (\pi_0,\pi_1)$ describes the ordering of the subintervals before and after the map $f$ is iterated (compare Figure~1). 
\begin{center}
	\begin{tikzpicture}[scale=1.0]
	\draw[thick] (0,0)--(8,0);     	
	\draw[thick] (0,0)--(0,0.25);     	
	\draw (0.75,0)--(0.75,0) node[above] {$1$};
	\draw[thick] (1.5,0)--(1.5,0.25);
	\draw (2.875,0)--(2.875,0) node[above] {$2$};
	\draw[thick] (4.25,0)--(4.25,0.25);  	
	\draw (5.125,0)--(5.125,0) node[above] {$3$};    	
	\draw[thick] (6,0)--(6,0.25);
	\draw (7,0)--(7,0) node[above] {$4$};  
	\draw[thick] (8,0)--(8,0.25);
	
	\draw[thick] (0,-0.25)--(8,-0.25);    	
	\draw[thick] (0,-0.25)--(0,-0.5);  
	\draw (1,-0.25)--(1,-0.25) node[below] {$4$}; 
	\draw[thick] (2,-0.25)--(2,-0.5);		
	\draw (2.875,-0.25)--(2.875,-0.25) node[below] {$3$};    	
	\draw[thick] (3.75,-0.25)--(3.75,-0.5);  	
	\draw (5.125,-0.25)--(5.125,-0.25) node[below] {$2$}; 
	\draw[thick] (6.5,-0.25)--(6.5,-0.5);
	\draw (7.25,-0.25)--(7.25,-0.25) node[below] {$1$}; 
	\draw[thick] (8,-0.25)--(8,-0.5); 
	
	\end{tikzpicture}\\    	
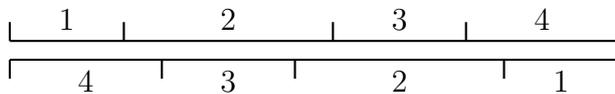
	Figure 1. Interval exchange transformation for $n=4$ \label{fig:1}
\end{center}

Whenever it is necessary to stress the number of subintervals involved, the map $f$ is called an $n$-interval exchange transformation or shorthand an $n$-IET. The combinatorial data is \textit{not} uniquely determined by $f$ (see e.g. \cite{Via06}, Example~1.3). In contrast, the expression $\rho = \pi_1 \circ \pi_0^{-1}$ is unique and called the \textbf{monodromy invariant} of $f$. When we normalize $\pi_0 = \textrm{Id}$, then $\pi_1$ coincides with the monodromy invariant. If the combinatorial data satisfies 
\begin{align} \label{eq4}
\pi_0^{-1}(\left\{1, \ldots, k \right\}) = \pi_1^{-1}(\left\{1, \ldots, k \right\})
\end{align} 
for some $k < n$, the interval exchange transformation splits into two interval exchange transformations of simpler combinatorics. The analysis of interval exchange transformations is therefore usually restricted to \textbf{admissible} combinatorial data, for which \eqref{eq4} does not hold for any $k < n$. Moreover, an interval exchange transformation $f$ satisfies the \textbf{Keane condition} if the orbits of the end points of the subintervals are infinite and as disjoint as possible, i.e. $f^m(\partial I_\alpha) \neq \partial I_\beta$ for all $m \geq 1$ and $\alpha, \beta \in \mathcal{A}$ with $\pi_0(\beta) \neq 1$. Finally, let us consider the set of discontinuities $D(f) = \left\{ \beta_1, \beta_2, \ldots \beta_m \right\}$ of an interval exchange transformation $f$. A finite sequence of points $x_1,x_2,\ldots,x_k$ is a \textbf{$f$-chain} if $x_1, x_k$ both belong to $D(f) \cup \left\{ 0 \right\}$ and $f(x_i) = x_{i+1}$. A \textbf{maximal} $f$-chain is an $f$-chain, which is not a proper subset of another $f$-chain. Now suppose that $x \in D(f)$ is non-periodic and the	initial point in the unique maximal $f$-chain of length $N(x)$ to which it belongs. If $f^{N(x)}$ is discontinuous at $x$, then $x$ is a \textbf{fundamental discontinuity} of $f$. Further details on interval exchange transformation can be found e.g. in \cite{Ber18}, \cite{Via06} and \cite{Yoc06}.

\paragraph{Low-Discrepancy Orbits.} For a map $f: [0,1) \to [0,1)$, the \textbf{orbit} of a point $x \in [0,1)$ is the sequence $(f^i(x))_{i=1}^\infty$. An orbit  is called a \textbf{low-discrepancy orbit} if it defines a low-discrepancy sequence. The simplest class of examples of interval exchange transformations with low-discrepancy orbits are rotations which satisfy the assumptions of Theorem \ref{thm:kronecker}. Moreover, Theorem~\ref{thm:n=3} yields a complete classification for the remaining case of $3$-IETs. Besides that, the only known examples of interval exchange transformations with low-discrepancy orbits stem from \cite{Wei19}: For arbitrary $n \in \NN$, let $L \in \NN, S \in \NN_0$ such that $L+S = n$ and choose $\beta$ as the positive solution of $L \beta + S \beta^2 = 1$. Then the monodromy invariant $\rho_{L,S}$ is specified by
\begin{align*}
& \rho_{L,S}(i) = i + 1, \quad i = 1, \ldots, L-1,\\
& \rho_{L,S}(L) = L+S,\\
& \rho_{L,S}(L+1) = 1,\\
& \rho_{L,S}(i) = i - 1, \quad i = L+2, \ldots L + S.
\end{align*}
and the length data by $\lambda_i = \beta$ for $i=1,\ldots, L$ and $\lambda_i = \beta^2$ for $i = L+1, \ldots, L+S$. The corresponding interval exchange transformation is denoted by $f_{L,S}$. The following result holds.
\begin{thm} (\cite{Wei19}, Corollary 3.9) \label{thm:IET_examples}
	If $L \geq S$, then the sequence $(f_{L,S}^i(x_0))_{i=0}^\infty$ is a low-discrepancy sequence for all $x_0 \in [0,1)$.
\end{thm}

\paragraph{Conjugation.} The set of all interval exchange transformations forms a group $\mathbb{G}$ under the operation of functional composition. Conjugacy classes and centralizers of $\mathbb{G}$ have recently been studied in \cite{Ber18} building on earlier work of \cite{Bos16}, \cite{Nov12} and \cite{Vor11}. Here, we aim to understand the effect of conjugation on the discrepancy of an orbit. Theorem~\ref{thm:main_thm}, which is of interest on its own, serves this purpose. It suffices to prove the second inequality mentioned therein because the first inequality follows from the second one by applying $f^{-1}$ to $P^*$.
\begin{proof}[Proof of Theorem~\ref{thm:main_thm}]
	Without loss of generality $n \geq 2$. Let $\lambda_1,\ldots,\lambda_n$ denote the length data of $f$ and define $\lambda_0 = 0$ and $\Lambda_{k} =  \sum_{i=0}^k \lambda_i$ for $k=0,\ldots,n$. We set
	$$D_{N,i}^*(P) := \sup_{\Lambda_{i-1} \leq b \leq \Lambda_{i}} \left| \frac{|P \cap [0,b)|}{N} - b \right|.$$
	Note that $D_{N,i}^*(P) \leq D_N^ *(P)$ for all $i$. The number of points of $P$ lying in interval $i$ is denoted by $\#_i$. Moreover, we interpret the monodromy invariant $\rho$ as permutation $\pi$ and let $\lambda_i^* = \lambda_{\pi^{-1}(i)}$ be the length data after permutation and set $\lambda^*_0 = 0$ and $\Lambda_k^* = \sum_{i=0}^k \lambda^*_i$ for $k=0,\ldots,n$. 
	By definition we have
	\begin{align*} D_N^*(P^*) = \sup_{0 \leq b \leq 1} \left| \frac{|P^* \cap [0,b)|}{N} - b \right|. \end{align*}
	Assume that the supremum is achieved in the $k$-th interval (after permutation). Hence
	\begin{align*} D_N^*(P^*) = \sup_{\Lambda_{k-1}^* \leq b \leq \Lambda_k^*} \left| \frac{|P^* \cap [0,b)|}{N} - b \right| \end{align*}
	For the sake of clarity and completeness we consider the case $k=n$ first. Then 
	\begin{align} \label{eq1}
		D_N^*(P^*) = \left| \frac{L^*}{N} - b^* \right|
	\end{align}
	with $L^* = \sum_{i \neq \pi^{-1}(k)} \#_i + L_k \leq N, 0 \leq L_k \leq \#_{\pi^{-1}(k)} $ and $b^* = \sum_{i \neq \pi^{-1}(k)} \lambda_i + b_k \leq 1, 0 \leq b_k \leq \lambda_{\pi^{-1}(k)}$. Thus we get
	\begin{align*} 
		D_N^*(P^*) & = \left| \frac{\sum\limits_{i \neq \pi^{-1}(k)} \#_i}{N} - \sum\limits_{i \neq \pi^{-1}(k)} \lambda_i + \frac{L_k}{N} - b_k \right|\\
		& \leq \left| \frac{\sum\limits_{i < \pi^{-1}(k)} \#_i}{N} - \sum_{i < \pi^{-1}(k)} \lambda_i + \frac{L_k}{N} - b_k \right| + \left| \frac{\sum\limits_{i > \pi^{-1}(k)} \#_i}{N} - \sum\limits_{i > \pi^{-1}(k)} \lambda_i\right|\\
		& \leq D_{N,\pi^{-1}(k)}^*(P) + \left| \frac{N - \sum\limits_{i \leq \pi^{-1}(k)} \#_i}{N} - (1 - \sum\limits_{i \leq \pi^{-1}(k)} \lambda_i) \right|\\
		& \leq D_{N,\pi^{-1}(k)}^*(P) + D_{N,\pi^{-1}(k)}^*(P) \leq 2 D_N^*(P) \leq nD_N^*(P).
	\end{align*}
	Now let $k < n$ be arbitrary and let $J = \left\{j_1 < j_2  < \ldots < j_{k-1}\right\}$ denote the set of indices with $\pi^{-1}(j_i) < \pi^{-1}(k)$. We use similar notation as in the case $k=n$, namely we assume \eqref{eq1} with $L^* = \sum_{J} \#_i + L_k \leq N, 0 \leq L_k \leq \#_{\pi^{-1}(k)} $ and $b^* = \sum_{J} \lambda_i + b_k \leq 1, 0 \leq b_k \leq \lambda_{\pi^{-1}(k)}$. Then
	{\small 
	\begin{align*} 
	D_N^*(P^*) & = \left| \frac{\sum\limits_{J} \#_i}{N} - \sum\limits_{J} \lambda_i + \frac{L_k}{N} - b_k \right|\\
	& \leq \left| \frac{\sum\limits_{i < \pi^{-1}(k)} \#_i}{N} - \sum_{i < \pi^{-1}(k)} \lambda_i + \frac{L_k}{N} - b_k \right| + \left| \frac{\sum\limits_{i > \pi^{-1}(k)} \#_i}{N} - \sum_{i > \pi^{-1}(k)} \lambda_i \right|\\
	& \qquad + \left| \frac{\sum\limits_{i \neq \pi^{-1}(k), i \notin J} \#_i}{N} - \sum_{i \neq \pi^{-1}(k), i \notin J} \lambda_i \right| \\
	& \leq 2 D_N^*(P) + \left| \frac{\sum\limits_{i \neq \pi^{-1}(k), i \notin J} \#_i}{N} - \sum_{i \neq \pi^{-1}(k), i \notin J} \lambda_i \right|
	\end{align*}}
	If $j_1 = 1, j_2 = 2, \ldots j_r = r$ and $j_{r+1} > r+1$ for $r \in \NN$ then
	{\small
	\begin{align*}
	\left| \sum\limits_{i \neq \pi^{-1}(k), i \notin J} \left( \frac{\#_i}{N} - \lambda_i \right) \right| & \leq  \left| \sum_{i=1}^r \left( \frac{\#_i}{N} - \lambda_i \right) \right| + \left| \sum\limits_{i \neq \pi^{-1}(k), i \notin J \setminus \{1,\ldots,r\}} \left( \frac{\#_i}{N} - \lambda_i \right) \right|\\
	& \leq D_N^*(P) + \left| \sum\limits_{i \neq \pi^{-1}(k), i \notin J \setminus \{1,\ldots,r\}} \left( \frac{\#_i}{N} - \lambda_i \right) \right|
	\end{align*}}
	and $J^* = J \setminus \left\{ 1,...,r \right\}$ has $r$ elements less than $J$. If $j_1 > 1$, then
	{\small
	\begin{align*}
		\left| \sum\limits_{i \neq \pi^{-1}(k), i \notin J} \left( \frac{\#_i}{N} - \lambda_i \right) \right| & \leq  \left| \sum_{i=1}^{j_1} \left( \frac{\#_i}{N} - \lambda_i \right) \right| + \left| \sum\limits_{i \neq \pi^{-1}(k), i \notin J \setminus \{j_1\}} \left( \frac{\#_i}{N} - \lambda_i \right) \right|\\
		& \leq D_N^*(P) + \left| \sum\limits_{i \neq \pi^{-1}(k), i \notin J \setminus \{j_1\}} \left( \frac{\#_i}{N} - \lambda_i \right) \right|
	\end{align*}}
	and $J^* = J \setminus \left\{ j_1 \right\}$ has one element less than $J$. In both of the cases it follows by induction on the number of elements in $J$ that
	$$\left| \sum\limits_{i \neq \pi^{-1}(k), i \notin J} \left( \frac{\#_i}{N} - \lambda_i \right) \right| \leq (k-1) D_N^*(P).$$
	In total, the calculation yields
	$$D_N^*(P^*) \leq (k + 1) D_N^*(P).$$
	Since $k < n$, the claim follows.
\end{proof}
In what follows we discuss applications of Theorem~\ref{thm:main_thm} and Corollary~\ref{cor:ld_conjugacy} respectively. If $f^i(y)_{i=1}^\infty$ is known to be a low-discrepancy sequence, then Corollary~\ref{cor:ld_conjugacy} can be applied directly by choosing $x = g(y)$. Furthermore, note that neither the number of intervals nor the permutations of the two interval exchange transformations have to coincide. If $f,g$ are two $2$-IETs, then they can both be interpreted as rotations of the circle and hence $(f(x_i))_{i=1}^\infty$ and $(gfg^{-1}(x_i))_{i=1}^\infty$  are equal. This also shows that the low-discrepancy sequence examples $f_{L,S}(x)$ cannot be generated by Kronecker sequences using conjugation. Moreover, Corollary~\ref{cor:ld_conjugacy} implies that low-discrepancy is a conjugacy class invariant if all orbits of $f$ are known to yield low-discrepancy sequences. From Theorem~\ref{thm:IET_examples}, we hence get.

\begin{cor} For any interval exchange transformation $g$ and any $x \in [0,1)$, the sequence $gf_{L,S}g^{-1}(x)$ is a low-discrepancy sequence if $L \geq S$. \end{cor}

Although the known examples of $n$-IETs with low-discrepancy orbits are non-trivial in the sense that they cannot be generated by rotations exclusively, the following example shows amongst others that the map $f_{2,2}$ is conjugate to a $3$-IET.

\begin{exa} \label{exa:oldform} Let $f_{2,2}$ be the 4-IET and let $g^{-1}$ be the circle rotation by the angle $z$. Here we consider the three special cases $z \in \left\{ \beta, 2\beta, 2\beta + \beta^2 \right\}$. If $z = \beta$, then $gf_{2,2}g^{-1}$ has length data $(\beta,\beta^2,\beta^2,\beta)$ and monodromy invariant $(3,4,2,1)$. By merging the third and the fourth interval (before rotation), we see that $gf_{2,2}g^{-1}$ can also be regarded as a $3$-IET with length data $(\beta,\beta^2,1/2)$ and monodromy invariant $(3,2,1)$. In the case $z = 2\beta$ we have $\lambda = (\beta^2,\beta,\beta,\beta^2)$ and $\rho = (3,1,4,2)$ after conjugation. This $4$-IET cannot be simplified to a $3$-IET. Finally, if $z= 2\beta + \beta^2$, then $gf_{2,2}g^{-1}$ can be represented by an $3$-IET with $\lambda = (1/2,\beta,\beta^2)$ and monodromy invariant $(3,2,1)$.
\end{exa}

In fact, being conjugate to a $3$-IET is not a special feature of the examples $f_{L,S}$ but the typical case for $4$-IETs. Whenever one of the following conditions is satisfied then (possibly) after conjugation by a rotation an $n$-IET can be represented by an IET with a smaller number of intervals involved:
\begin{align*}
(1) \quad & \rho(i+1) = \rho(i) + 1, \quad \textrm{for some } 1 \leq j \leq n-1,\\
(2) \quad & \rho(i) = n, \rho(i+1) = 1, \quad \textrm{for some } 1 \leq j \leq n-1,\\
(3) \quad & \rho(n) = j, \rho(1) = j+1,\\
(4) \quad & \rho(n) = n, \rho(1) = 1.
\end{align*}
Generalizing notation from \cite{Ber17}, we call the monodromy invariant \textbf{strongly separating} if it does not fulfill any of the properties $(1)-(4)$. The following theorem shows that there is only one admissible monodromy invariant of a $4$-IET which cannot be achieved by conjugation from a $3$-IET (compare Figure~1).

\begin{thm} \label{thm:old_forms_n4} Let $f$ define an admissible $4$-IET with monodromy invariant $\rho \neq (4,3,2,1)$. Then $f$ is conjugate to a $3$-IET.
\end{thm}

\begin{proof} We consider the $13$ admissible monodromy invariants of 4-IETs. These are
	\begin{align*}
	& (4,3,2,1), \quad (4,1,3,2), \quad (3,1,4,2), \quad (4,2,1,3), \quad (2,4,3,1),\\
	& (3,2,4,1), \quad (2,4,1,3), \quad (4,2,3,1), \quad (4,1,2,3), \quad (4,3,1,2),\\
	& (3,4,1,2), \quad (2,3,4,1), \quad (3,4,2,1). 
	\end{align*}
	The only strongly separating monodromy invariant is $\rho = (4,3,2,1)$. Following the lines of the proof of Proposition~2.3 in \cite{Ber17}, the map $f$ is hence conjugate to a $2$- or $3$-IET because the number of discontinuities after an appropriate rotation is at most $3$.
\end{proof}

Summing up, we therefore have a criterion at hand (Theorem~\ref{thm:kronecker}, Theorem~\ref{thm:n=3}) to decide, if a given interval exchange transformation with $n=4$ intervals and monodromy invariant $\rho \neq (4,3,2,1)$ has a low-discrepancy orbit or not.

\begin{defi} We call an $n$-IET $h$ which is given by $h=gfg^{-1}$ with $f$ being an $m$-IET with $m < n$ and $g$ an arbitrary interval exchange transformation an \textbf{old transformation}. Otherwise $h$ is called a \textbf{new transformation}.
\end{defi}

Thus, Theorem~\ref{thm:old_forms_n4} can be restated in the form that every $4$-IET with $\rho \neq (4,3,2,1)$ is an old transformation. For $n=5$, there are $21$ of $71$ admissible monodromy invariants that can potentially yield new transformations and for $n=6$, there are $126$ out of $461$. Note that our notation of strong separation only takes into account conjugation by rotations and therefore not all of the identified potential new transformations are truly new. Indeed, the following example shows that not every interval exchange transformation that is strongly separating is necessarily a new transformation.
\begin{exa} \label{exa:strongly_separating} Let $f$ be a 2-IET with $\rho_f = (2,1)$ and $g$ be a $3$-IET with $\rho_g = (3,2,1)$. We choose $\beta$ as the unique positive solution of $2\beta + 2\beta^2 = 1$ and let $\lambda_f = (1-\beta,\beta)$ and $\lambda_g = (\beta,\beta^2,1-\beta-\beta^2)$. Then $gfg^{-1}$ has length data $(\beta,\beta^2,\beta^2,\beta^2,\beta-\beta^2)$ and monodromy invariant $\rho_{gfg^{-1}} = (4,2,5,3,1)$ which is a strongly separating IET.
\end{exa}
Moreover, the interval exchange transformation $f_{L,S}$ is an old transformation for any choice $L \in \mathbb{N}, S \in \mathbb{N}_0$: by counting discontinuities as in Example~\ref{exa:oldform}, every map $f_{L,S}$ can be seen to be conjugate to a $3$-IET with monodromy invariant $\rho = (3,2,1)$ and length data $\lambda_1 = \beta, \lambda_2 = (L-1)\beta, \lambda_3 = S\beta^2$. This fact suffices to generalize Theorem~\ref{thm:IET_examples} and leave away the condition $L \geq S$ therein.
\begin{thm} Let $L \in \mathbb{N}, S \in \mathbb{N}_0$ and let $\beta > 0$ be the positive solution of $L\beta + S\beta^2 = 1$. Then the sequence $(f_{L,S}^i(x_0))_{i=0}^\infty$ is a low-discrepancy sequence for all $x_0 \in [0,1)$ if and only if $\beta$ is irrational.
\end{thm}  
\begin{proof} If $\beta$ is rational then $f_{L,S}$ has finite order and cannot have a low-discrepancy orbit. If $\beta$ is irrational, then 
	$$\nu = \frac{\lambda_2+\lambda_3}{1+ \lambda_2} = \frac{1-\beta}{1+(L-1)\beta}$$
is a real algebraic number of degree $2$. In particular, $\nu$ has bounded partial quotients and thus also its Ces\`aro mean is bounded. Therefore, the claim follows from Theorem~\ref{thm:n=3}.
\end{proof}

In the case of monodromy invariant $\rho = (4,3,2,1)$ we finally give sufficient conditions under which $f$ is a new transformation.

\begin{prop} \label{prop:newtrans} Let $f$ be an arbitrary $4$-IET with monodromy invariant $(4,3,2,1)$ which satisfies the Keane condition. Furthermore assume that $f^2$ is discontinuous at $f^{-1}(0)$. Then $f$ is a new transformation. \end{prop}

To prove Proposition~\ref{prop:newtrans}, we use two results from \cite{Ber18} (Proposition~4.3, Corollary~4.5) which we combine here to get the following lemma.

\begin{lem} (\cite{Ber18}) \label{lem:ber} Let $f,g$ be two interval exchange transformations. Then $f$ and $gfg^{-1}$ have the same number of fundamental discontinuities.
\end{lem}

\begin{proof}[Proof of Proposition~\ref{prop:newtrans}] Suppose that $f = ghg^{-1}$. If $f$ was an old transformation, then $h$ would either have monodromy invariant $\rho = (2,1)$ or $\rho = (3,2,1)$ because all other $3$-IETs are either not admissible or a rotation. By Proposition~4.6 in \cite{Ber18}, the IET $h$ can have at most three fundamental discontinuities. However since $f^2$ is discontinuous at $f^{-1}(0)$, the IET $f$ has four fundamental discontinuities. Hence $f$ and $h$ cannot be conjugate by Lemma~\ref{lem:ber} and $f$ cannot be an old transformation. 
\end{proof}
More generally, every $n$-IET which satisfies the Keane condition and has $n$ fundamental discontinuities is a new transformation by Corollary~4.5 in \cite{Ber18}.

\paragraph{Acknowledgement.} I would like to thank the anonymous referee for his careful reading and valuable comments.

\bibliographystyle{acm}
\bibdata{references}
\bibliography{references}

\textsc{Hochschule Ruhr West, Duisburger Str. 100, D-45479 M\"ulheim an der Ruhr}\\
\textit{E-mail address:} \texttt{christian.weiss@hs-ruhrwest.de}

\end{document}